\pdfoutput=1

\documentclass[11pt]{article}
\usepackage{amsfonts}
\usepackage{amssymb}
\usepackage{fancyhdr}
\usepackage{comment}
\usepackage{amsmath}
\usepackage{amsthm}
\usepackage{graphicx}
\usepackage{enumerate}
\usepackage[hypcap]{caption}
\usepackage[nodayofweek,level]{datetime}
\usepackage{color, soul}
\usepackage[T1]{fontenc}
\usepackage[utf8]{inputenc}
\usepackage{authblk}
\newtheorem{theorem}{Theorem}
\newtheorem{remark}{Remark}
\newtheorem{lemma}{Lemma}

\newtheorem{assumption}{Assumption}

\usepackage{etoolbox}
\makeatletter
\patchcmd{\l@section}
  {\hfil}
  {\leaders\hbox{\normalfont$\m@th\mkern \@dotsep mu\hbox{.}\mkern \@dotsep mu$}\hfill}
  {}{}
\makeatother

\title{\Large\bf Posterior Asymptotic Normality for an Individual Coordinate in High-dimensional Linear Regression}
\author{\vspace{-1ex} \textsc{\small Dana Yang}\thanks{xiaoqian.yang@yale.edu}}
\affil{\vspace{-1ex} \small \it Department of Statistics and Data Science, Yale University}
\date{\vspace{-9ex}}

\begin{document}

\maketitle

\begin{abstract}
We consider the sparse high-dimensional linear regression model $Y=Xb+\epsilon$ where $b$ is a sparse vector. For the Bayesian approach to this problem, many authors have considered the behavior of the posterior distribution when, in truth, $Y=X\beta+\epsilon$ for some given $\beta$. There have been numerous results about the rate at which the posterior distribution concentrates around $\beta$, but few results about the shape of that posterior distribution. We propose a prior distribution for $b$ such that the marginal posterior distribution of an individual coordinate $b_i$ is asymptotically normal centered around an asymptotically efficient estimator, under the truth. Such a result gives Bayesian credible intervals that match with the confidence intervals obtained from an asymptotically efficient estimator for $b_i$. We also discuss ways of obtaining such asymptotically efficient estimators on individual coordinates. We compare the two-step procedure proposed by Zhang and Zhang~\cite{Zhang2014confidence} and a one-step modified penalization method. 
\end{abstract}

\section{Introduction.}

Consider the regression model
\begin{equation}\label{model}
Y=Xb+\epsilon,\;\;\;\epsilon\sim\mathcal{N}(0,I_n).
\end{equation}
The design matrix $X$ is of dimension $n\times p$. We are particularly interested in the case where $p>n$, for which $b$ itself is not identifiable. In such a setting identifiability can be attained by adding a sparsity constraint on $|b|_0$, the number of nonzero $b_i$'s. That is, the model consists of a family of probability measures $\{\mathbb{P}_b:b\in\mathbb{R}^p,|b|_0\leq s^*\}$, and the observation $Y$ is distributed $\mathcal{N}(Xb,I_n)$ under $\mathbb{P}_b$.

\medskip

We are interested in the Bayesian inference on the vector $b$, when $Y$ is actually distributed $N(X\beta, I_n)$ for some truth $\beta$. If $p$ were fixed and $X$ were full rank, classical theorems (the Bernstein-von Mises theorem, as in~\cite[page~141]{van2000asymptotic}) gives conditions under which the posterior distribution of $b$ is asymptotically normal centered at the least squares estimator, with variance $(X^T X)^{-1}$ under $\mathbb{P}_\beta$.

\medskip

The classical theorem fails when $p>n$. Although sparse priors have been proposed that give good posterior contraction rates~\cite{castillo2015bayesian}~\cite{gao2015general}, posterior normality of $b$ is only obtained under strong signal-to-noise ratio (SNR) conditions, such as the SNR conditions of Castillo el al.~\cite[Corollary 2]{castillo2015bayesian}, which forced the posterior to eventually have the same support as $\beta$. Effectively, their conditions reduce the problem to the classical, fixed dimensional case. However that is not the most interesting scenario. Without the SNR condition, Castillo et al.~\cite[Theorem 6]{castillo2015bayesian} pointed out that under the sparse prior, the posterior distribution of $b$ behaves like a mixture of Gaussians.

\medskip

However, there is hope to obtain posterior normality results without the SNR condition if one considers the situation where only one component of $b$ is of interest, say $b_1$, without loss of generality. All the other components are viewed as nuisance parameters. As shown by Zhang and Zhang~\cite{Zhang2014confidence} in a non-Bayesian setting, it is possible to construct estimators that are efficient in the classical sense that
\begin{equation}\label{beta1.hat}
\hat{\beta}_1=\beta_1+\frac{X_1^T\epsilon}{|X_1|^2}+o_p\left(\frac{1}{\sqrt{n}}\right).
\end{equation}
We will use $o_p(\cdot)$ as a short hand for a stochastically small order term under $\mathbb{P}_\beta$ throughout this document. Here $X_i$ denotes the $i$'th column of $X$, and the $o_p(\cdot)$ indicates that a term is of stochastically smaller order under $\mathbb{P}_\beta$. Later we also write$X_{-i}$ to denote the $n\times (p-1)$ matrix formed by all columns of $X$ except for $X_i$. The $|\cdot|$ norm on a vector refers to the Euclidean norm. 

\medskip

Approximation~\eqref{beta1.hat} is useful when $|X_1|$ is of order $\sqrt{n}$, in which the expansion~\eqref{beta1.hat} implies weak convergence~\cite[page~171]{pollard2002user}:
\begin{equation*}
|X_1|(\hat{\beta}_1^{(ZZ)}-\beta_1)\leadsto \mathcal{N}(0,1)
\end{equation*}
under $\mathbb{P}_\beta$ (Such behavior for $|X_1|$ is obtained with high probability when $X$ is generated {\it i.i.d.} from the standard normal distribution). More precisely, Zhang and Zhang~\cite{Zhang2014confidence} proposed a two-step estimator that satisfies~\eqref{beta1.hat} under some regularity assumptions on $X$ and no SNR conditions. They required the following behavior for $X$.

\medskip

\begin{assumption}\label{assump.id}
Let $\gamma_i=X_1^TX_i/|X_1|^2, \lambda_n=\sqrt{\frac{\log p}{n}}$. There exists a constant $c_1>0$ for which
\begin{equation*}
\max_{2\leq i\leq p}\gamma_i\leq c_1\lambda_n.
\end{equation*}
\end{assumption}

\begin{assumption}\label{assump.rec} (REC($3s^*,c_2$))
There exists constants $c_2,c'>0$ for which
\begin{equation}\label{def.rec}
\kappa(3s^*,c_2)=\min_{\substack{J\subset[p],\\ |J|\leq 3s^*}}\inf_{\substack{b\neq 0, \\ \left|b_{J^C}\right|_1\leq c_2|b_J|_1}}\frac{|Xb|}{\sqrt{n}|b_J|}>c'>0.
\end{equation}
\end{assumption}

\begin{assumption}\label{assump.dim}
The model dimension satisfies
\begin{equation*}
s^*\log p=o(\sqrt{n}).
\end{equation*}
\end{assumption}

\begin{remark}
Assumption~\ref{assump.rec} is known as the restricted eigenvalue condition~\cite[page~1710]{bickel2009simultaneous} required for penalized regression estimators such as the LASSO estimator~\cite[page~1]{tibshirani1996regression} and the Dantzig selector~\cite[page~1]{candes2007dantzig} to enjoy optimal $l^1$ and $l^2$ convergence rates. 
\end{remark}

\begin{theorem}\label{ZZ.th}~\cite[Section 2.1,3.1]{Zhang2014confidence}
Under assumptions~\ref{assump.id},~\ref{assump.rec} and~\ref{assump.dim}, the estimator $\hat{\beta}_1^{(ZZ)}$ has expansion~\eqref{beta1.hat}. 
\end{theorem}
The exact form of the estimator $\hat{\beta}_1^{(ZZ)}$ will be given in section~\ref{sec:freq.debias}. 

\medskip

The goal of this paper is to give a Bayesian analogue for Theorem~\ref{ZZ.th}, in the form of a prior distribution on $b$ such that as $n,p\rightarrow\infty$, the posterior distribution of $b_1$ starts to resemble a normal distribution to centered around an estimator in the form of~\eqref{beta1.hat}. Note that the sparse prior introduced by Castillo et al.~\cite{castillo2015bayesian} does not meet our goal since the marginal posterior distribution of $b_1$ under the sparse prior converges weakly to a mixture of normal distributions without consistent model selection. 

\begin{theorem}\label{main.th}
Under assumptions~\ref{assump.id},~\ref{assump.rec} and~\ref{assump.dim} and the constraint $|X_1|=O(\sqrt{n})$, there exists a prior on $b$ for which the posterior distribution of $|X_1|(b_1-\hat{\beta}_1)$ satisfies
\begin{equation}\label{goal}
\left\lVert\mathcal{L}\left(|X_1|(b_1-\hat{\beta}_1)|Y\right)-\mathcal{N}(0,1)\right\rVert_{BL}\rightarrow 0 \text{ in }\mathbb{P}_\beta,
\end{equation}
where $\hat{\beta}_1$ is an estimator of $\beta_1$ with expansion~\eqref{beta1.hat}.
\end{theorem}
The measure used here to quantify the discrepancy between probability measures is the bounded-Lipschitz metric~\cite[page~1]{dudley1972speeds}. The convergence of a sequence of distributions to a fixed distribution in bounded-Lipschitz metric is equivalent to weak convergence. 

\section{The prior and its background stories.}
\subsection{How does de-biasing work?}\label{sec:freq.debias}
In sparse linear regression, penalized likelihood estimators such as the LASSO are often used and tend to give good global properties. One desirable property is the following bound on the $l_1$ loss.
\begin{equation}\label{eq:l1.control}
\mathbb{P}_\beta\left\{|\tilde{\beta}-\beta|_1>Cs^*\lambda_n\right\}\rightarrow 0\text{ as }n,p\rightarrow \infty\text{ for some }C>0,
\end{equation}
where $\lambda_n$ is as defined in assumption~\ref{assump.id}. For example, Bickel et al.~\cite[Theorem~7.1]{bickel2009simultaneous} showed that under the REC condition (assumption~\ref{assump.rec}) the LASSO estimator satisfies~\eqref{eq:l1.control}. 

\medskip

In general, penalized likelihood estimators introduce bias for the estimation of individual coordinates. To eliminate this bias, Zhang and Zhang~\cite{Zhang2014confidence} proposed a two-step procedure. First find a $\tilde{\beta}$, perhaps via a LASSO procedure that satisfies~\eqref{eq:l1.control}. Then define
\begin{equation*}
\hat{\beta}_1^{(ZZ)}=\arg\min_{b_1\in \mathbb{R}}|Y-X_{-1}\tilde{\beta}_{-1}-b_1X_1|^2,
\end{equation*}
The idea behind this estimator is to penalize the magnitude of all coordinates except the one of interest. Under assumptions~\ref{assump.id},~\ref{assump.rec} and~\ref{assump.dim}, the one-step estimator $\hat{\beta}_1^{(ZZ)}$ is asymptotically unbiased with expansion~\eqref{beta1.hat}. The same asymptotic behavior can be obtained in a single step, as in the next theorem. The idea of penalizing all coordinates but one to eliminate the bias is seen more clearly here.

\begin{theorem}\label{th.one.step}
Define
\begin{equation*}
\hat{\beta}^{(one-step)}=\arg\min_{b\in \mathbb{R}^p}\left(|Y-Xb|_2^2+\eta_n\sum_{i\geq 2}|b_i|\right).
\end{equation*}
Choose $\eta_n$ to be a large enough multiple of $n\lambda_n$. Under assumptions~\ref{assump.id},~\ref{assump.rec} and~\ref{assump.dim}, the one-step de-biasing estimator of $\beta_1$ achieves $l_1$ control~\eqref{eq:l1.control} and de-biasing simultaneously. The estimator for the first coordinate satisfies
\begin{equation*}
\hat{\beta}_1^{(one-step)}=\beta_1+\frac{X_1^T\epsilon}{|X_1|^2}+o_p\left(\frac{1}{\sqrt{n}}\right).
\end{equation*}
\end{theorem}
\begin{proof}
In the proof of theorem~\ref{th.one.step} we will refer to the one step estimator as $\hat{\beta}$. We will first show that $\hat{\beta}$ satisfies~\eqref{eq:l1.control}. We know that when the penalty involves all coordinates of $b$, then the bound on the $l_1$ norm is true~\cite[Theorem~7.1]{bickel2009simultaneous}. It turned out that leaving one term out the of penalty does not ruin that property.

\medskip

As in the proof of~\cite[Theorem~7.1]{bickel2009simultaneous}, we compare the evaluation of the penalized likelihood function at $\hat{\beta}$ and the truth $\beta$ using the definition of $\hat{\beta}$.
\begin{equation*}
|Y-X\hat{\beta}|_2^2+\eta_n|\hat{\beta}_{-1}|_1\leq |Y-X\beta|_2^2+\eta_n|\beta_{-1}|_1.
\end{equation*}
Plug in $Y=X\beta+\epsilon$, the above is reduced to
\begin{equation*}
|X(\hat{\beta}-\beta)|_2^2\leq 2\sum_{i\leq n}\xi_i(\hat{\beta}_i-\beta_i)+\eta_n(|\beta_{-1}|_1-|\hat{\beta}_{-1}|_1),
\end{equation*}
where $\xi_i=X_i^T\epsilon$. With high probability $|\max_{i\leq n}\xi_i|\leq R=C_2 n\lambda_n$, in which case we have
\begin{equation}\label{compare.ineq}
|X(\hat{\beta}-\beta)|_2^2\leq 2R|\hat{\beta}-\beta|_1+\eta_n(|\beta_{-1}|_1-|\hat{\beta}_{-1}|_1).
\end{equation}
From here we need to discuss two situations. First consider the case where $1$ is in the $S$, the support of $\beta$. The expression above is bounded by
\begin{equation*}
(2R+\eta_n)|(\hat{\beta}-\beta)_S|_1+(2R-\eta_n)|(\hat{\beta}-\beta)_{S^C}|_1.
\end{equation*}
By choosing $\eta_n$ to be a large enough multiple of $n\lambda_n$, we have
\begin{equation*}
|X(\hat{\beta}-\beta)|_2^2\leq c_1n\lambda_n|(\hat{\beta}-\beta)_S|_1-c_2n\lambda_n|(\hat{\beta}-\beta)_{S^C}|_1.
\end{equation*}
Since the lefthand side is nonnegative, the above implies
\begin{equation}\label{cone.1}
|(\hat{\beta}-\beta)_{S^C}|_1\leq \frac{c_1}{c_2}|(\hat{\beta}-\beta)_S|_1.
\end{equation}
Therefore under assumption REC($c_0/c_1$, $\kappa$), we can further bound the prediction loss by
\begin{align*}
& c_1\sqrt{s^*n\log p}\cdot|(\hat{\beta}-\beta)_S|_2\\
\leq & \frac{c_1}{\kappa}\sqrt{s^*\log p}\cdot|X(\hat{\beta}-\beta)|_2.
\end{align*}
So far we have shown with high probability,
\begin{equation*}
|X(\hat{\beta}-\beta)|_2\leq \frac{c_1}{\kappa}\sqrt{s^*\log p}.
\end{equation*}
Under the REC assumption, we can go back to bound the $l1$ loss.
\begin{equation*}
|(\hat{\beta}-\beta)_S|_1\leq \sqrt{s^*}|(\hat{\beta}-\beta)_S|_2\leq \frac{1}{\kappa}\sqrt{\frac{s^*}{n}}|X(\hat{\beta}-\beta)|_2\leq \frac{c_1}{\kappa^2}s^*\lambda_n.
\end{equation*}
Therefore with~\eqref{cone.1} we have
\begin{equation*}
|\hat{\beta}-\beta|_1\leq \left(1+\frac{c_1}{c_2}\right)\frac{c_1}{\kappa^2}s^*\lambda_n.
\end{equation*}
The proof for the other case turned out to be messier. But the general idea remains the same. When $1\in S^C$, we can bound the RHS of~\eqref{compare.ineq} by
\begin{equation*}
(2R+\eta_n)\left|(\hat{\beta}-\beta)_{S\cup \{1\}}\right|_1+(2R-\eta_n)\left|(\hat{\beta}-\beta)_{S^C\char`\\ \{1\}}\right|_1,
\end{equation*}
Choosing $\lambda$ to be a large multiple of $\sqrt{n\log p}$ as in the $1\in S$ case, we have
\begin{equation*}
|X(\hat{\beta}-\beta)|_2^2\leq c_1n\lambda_n|(\hat{\beta}-\beta)_{S\cup\{1\}}|_1-c_2n\lambda_n|(\hat{\beta}-\beta)_{S^C\char`\\ \{1\}}|_1,
\end{equation*}
which implies
\begin{equation}\label{cone.2}
|(\hat{\beta}-\beta)_{S^C\char`\\ \{1\}}|_1\leq \frac{c_1}{c_2}|(\hat{\beta}-\beta)_{S\cup\{1\}}|_1.
\end{equation}
Again use assumption~\ref{assump.rec} to deduce the $l_1$ control~\eqref{eq:l1.control}.

Observe that the penalty term does not involve $b_1$.
\begin{align}
\nonumber\hat{\beta}_1= & \arg\min_{b_1\in\mathbb{R}}|Y-X_{-1}\hat{\beta}_{-1}-b_1X_1|_2^2\\
\label{beta1hat.rep}= & \beta_1+\sum_{i\geq 2}\gamma_i(\beta_i-\hat{\beta}_i)+\frac{X_1^T\epsilon}{|X_1|^2}.
\end{align}
We only need to show the second term in~\eqref{beta1hat.rep} is of order $o_p(1/\sqrt{n})$. Bound the absolute value of that term with
\begin{equation*}
\max_{i\geq 2}|\gamma_i|\cdot|\hat{\beta}_S-\beta_S|_1\leq \left(c_1\lambda_n\right)\left(C_1s^*\lambda_n\right),
\end{equation*}
by assumption~\ref{assump.id} and the $l_1$ control~\eqref{eq:l1.control}. That is then bounded by $O_p(s^*\lambda_n^2)=o_p(1/\sqrt{n})$ by assumption~\ref{assump.dim}.

\end{proof}

\begin{remark}
With some careful manipulation the REC($3s^*, c_2$) condition as in assumption~\ref{assump.rec} can be reduced to REC($s^*,c_2$). The proof would require an extra step of bounding $|\hat{\beta}_1-\beta_1|$ by $o_p(|\hat{\beta}_S-\beta_S|_1)+O_p(1/\sqrt{n})$.
\end{remark}

The ideas in the proofs for the two de-biasing estimators $\hat{\beta}_1^{(ZZ)}$ and $\hat{\beta}_1^{(one-step)}$ are similar. Ideally we want to run the regression
\begin{equation}\label{eq:ideal.reg}
\arg\min_{b_1\in\mathbb{R}}|Y-X_{-1}\beta_{-1}-b_1X_1|^2.
\end{equation}
That gives a perfectly efficient and unbiased estimator. However $\beta_{-1}$ is not observed. It is natural to replace it with an estimator which is made globally close to the truth $\beta_{-1}$ using penalized likelihood approach. As seen in the proof of Theorem~\ref{th.one.step}, most of the work goes into establishing global $l_1$ control~\eqref{eq:l1.control}. The de-biasing estimator is then obtained by running an ordinary least squares regression like~\eqref{eq:ideal.reg}, replacing $\beta_{-1}$ by some estimator satisfying~\eqref{eq:l1.control}, so that the solution to the least squares optimization is close to the solution of~\eqref{eq:ideal.reg} with high probability.

\subsection{Bayesian analogue of de-biasing estimators.}
We would like to give a Bayesian analogue to the de-biasing estimators discussed above. As pointed out in the last section, it is essential to establish $l_1$ control on the vector $b_{-1}-\beta_{-1}$. Castillo et al.~\cite{castillo2015bayesian} and Gao et al.~\cite{gao2015general} have proposed priors that penalize sparsity of submodel dimension and provided theoretical guarantees such as LASSO-type contraction rates under the posterior distribution. This is the prior construction of Gao et al.~\cite[Section 3]{gao2015general}. 
\begin{enumerate}
\item The size $s$ of the dimension of the sub-model in the direction orthogonal to $X_1$ has probability mass function $\pi(s)\propto \exp(-Ds\log p)$.

\item $S|s\sim Unif(Z_s:=\{S\subset \{1,...,p\}:|S|=s,X_S\text{ is full rank}\})$.

\item Given the subset selection $S$, the coefficients $b_S$ has density $f_S(b_S)\propto\exp(-\eta_n |X_Sb_S|)$ for suitably chosen $\eta_n$.
\end{enumerate}

Gao et al.~\cite{gao2015general} gave conditions under which we have a good $l_1$ posterior contraction rate.
\begin{lemma}\label{lem:post.l1}(Corollary 5.4,~\cite{gao2015general})
If the design matrix $X$ satisfies
\begin{equation}\label{eq:comp}
\kappa_0((2+\delta)s^*,X)=\inf_{|b|_0\leq (2+\delta)s^*}\frac{\sqrt{s^*}|Xb|}{\sqrt{n}|b|_1}\geq c
\end{equation}
for some positive constant $c,\delta$, then there is constant $c_3>0$ and large enough $D>0$ for which
\begin{equation*}
\mathbb{P}_\beta \mu_Y\left\{|b-\beta|_1>c_3s^*\lambda_n\right\}\rightarrow 0,
\end{equation*}
\end{lemma}

We slightly modify the sparse prior of Gao et al.~\cite{gao2015general} to give good, asymptotically normal posterior behavior for a single coordinate. As we discussed in the last section, classical approaches to de-biasing exploit the idea of penalizing all coordinates except the one of interest. Our prior construction mimics that idea by putting the sparse prior only on $b_{-1}$.

\subsection{The prior.}
Denote the matrix projecting $\mathbb{R}^n$ to $span(X_1)$ by $H$. Under the model where $Y\sim\mathcal{N}(Xb,I_n)$, the likelihood function has the factorization
\begin{align*}
\mathcal{L}_n(b)= & \frac{1}{\sqrt{n}(2\pi)^{n/2}}\exp\left(-\frac{|Y-Xb|^2}{2}\right)\\
= & \frac{1}{\sqrt{n}(2\pi)^{n/2}}\exp\left(-\frac{|HY-HXb|^2}{2}\right)\exp\left(-\frac{|(I-H)Y-(I-H)Xb|^2}{2}\right).
\end{align*}
Write $W=(I-H)X_{-1}$ and reparametrize $b_1^*=b_1+\sum_{i\geq 2}\gamma_i b_i$ with $\gamma_i$ as defined in assumption~\ref{assump.id}. The likelihood $\mathcal{L}_n(b)$ can be rewritten as a constant multiple of
\begin{equation*}
\exp\left(-\frac{|HY-b_1^*X_1|^2}{2}\right)\exp\left(-\frac{|(I-H)Y-W b_{-1}|^2}{2}\right).
\end{equation*}
The likelihood factorizes into a function of $b_1^*$ and $b_{-1}$. Therefore if we make $b_1^*$ and $b_{-1}$ independent under the prior, they will be independent under the posterior. We put a Gaussian prior on $b_1^*$ to mimic the ordinary least square optimization step in the classical approaches. We put a sparse prior analogue to that of Gao et al.~\cite[section~3]{gao2015general} on $b_{-1}$, using $W$ as the design matrix in the prior construction. By lemma~\ref{lem:post.l1}, $b_{-1}$ is close to $\beta_{-1}$ in $l_1$ norm with high posterior probability as long as $\kappa_o((2+\delta)s^*,W)$ is bounded away from 0.

\medskip

We make $b_1^*$ and $b_{-1}$ independent under the prior distribution. The product distribution corresponds to a prior distribution on the original vector $b$. Note that under the prior distribution $b_1$ and $b_{-1}$ are not necessarily independent.

\medskip

This modified prior also has the effect of eliminating a bias term, in a fashion analogues to that of the two-step procedure $\hat{\beta}_1^{(ZZ)}$. The joint posterior distribution of $b_1^*$ and $b_{-1}$ factorizes into two marginals. In the $X_1$ direction, the posterior distribution of $b_1^*$ is asymptotically Gaussian centered around $\frac{X_1^TY}{|X_1|^2}=\beta_1^*+\frac{X_1^T\epsilon}{|X_1|^2}$. After we reverse the reparametrization we want the posterior distribution of $b_1$ to be asymptotically Gaussian centered around an efficient estimator $\hat{\beta}_1=\beta_1+\frac{X_1^T\epsilon}{|X_1|^2}+o_p(1/\sqrt{n})$. Therefore we need to show $b_1^*-b_1$ is very close to $\beta_1^*-\beta_1$. That can be obtained from the $l_1$ control on $b_{-1}-\beta_{-1}$ under the posterior. In the next section we will give the proof to our main posterior asymptotic normality result (Theorem~\ref{main.th}) in detail.

\section{Proof of Theorem~\ref{main.th}.}

Since that prior and the likelihood of $b_1^*$ are both Gaussian, we can work out the exact posterior distribution.
\begin{equation*}
b_1^*|Y\sim \mathcal{N}\left(\frac{\sigma_n^2}{1+|X_1|^2\sigma_n^2}X_1^T Y, \frac{\sigma_n^2}{1+|X_1|^2\sigma_n^2}\right).
\end{equation*}
Since $b_1^*$ and $b_{-1}$ are independent under the posterior distribution, the above is also the distribution of $b_1^*$ given $Y$ and $b_{-1}$. That implies the distribution of $|X_1|(b_1-\hat{\beta}_1)$ given $Y$ and $b_{-1}$ is
\begin{equation}\label{post.b1}
\mathcal{N}\left(|X_1|\left(\frac{\sigma_n^2}{1+|X_1|^2\sigma_n^2}X_1^T Y-\sum_{i\geq 2}\gamma_i b_i-\hat{\beta}_1\right), \frac{\sigma_n^2|X_1|^2}{1+|X_1|^2\sigma_n^2}\right).
\end{equation}
Note that without conditioning on $b_{-1}$, the posterior distribution of $b_1$ is not necessarily Gaussian.

\medskip

The goal is to show the bounded-Lipschitz metric between the posterior distribution of $b_1$ and $\mathcal{N}(\hat{\beta}_1,1/|X_1|^2)$ goes to 0 under the truth. From Jensen's inequality and the definition of the bounded-Lipschitz norm we have
\begin{align*}
& \left\lVert\mathcal{L}(|X_1|(b_1-\hat{\beta}_1)|Y)-\mathcal{N}(0,1)\right\rVert_{BL}\\
\leq & \mu_Y^{b_{-1}}\left\lVert\mathcal{L}(|X_1|(b_1-\hat{\beta}_1)|Y,b_{-1})-\mathcal{N}(0,1)\right\rVert_{BL}.
\end{align*}
For simplicity denote the posterior mean and variance in~\eqref{post.b1} as $\nu_n$ and $\tau_n^2$ respectively. The bounded-Lipschitz distance between two normals $\mathcal{N}(\mu_1,\sigma_1^2)$ and $\mathcal{N}(\mu_2,\sigma_2^2)$ is bounded by $(|\mu_1-\mu_2|+|\sigma_1-\sigma_2|)\wedge 2$. Hence the above is bounded by
\begin{equation*}
\mu_Y^{b_{-1}}\left(|\nu_n|\wedge 2\right)+\mu_Y^{b_{-1}}\left((|\tau_n-1|)\wedge 2\right).
\end{equation*}
Therefore to obtain the desired convergence in~\eqref{goal}, we only need to show

\begin{equation}\label{mean.conv}
\mathbb{P}_\beta \mu_Y^{b_{-1}}\left(|\nu_n|\wedge 2\right)\rightarrow 0,\;\;\;\text{and}
\end{equation}
\begin{equation}\label{var.conv}
\mathbb{P}_\beta \mu_Y^{b_{-1}}\left((|\tau_n-1|)\wedge 2\right)\rightarrow 0.
\end{equation}
To show~\eqref{mean.conv}, notice that the integrand is bounded. Hence it is equivalent to show convergence in probability. Write
\begin{align}
\nonumber|\nu_n|= & \frac{\sigma_n^2|X_1|^d}{1+\sigma_n^2|X_1|^2}\left(\beta_1+\frac{X_1^T\epsilon}{|X_1|^2}+\sum_{i\geq 2}\gamma_i\beta_i\right)\\
\nonumber& -\sum_{i\geq 2}\gamma_ib_i-|X_1|\left(\beta_1+\frac{X_1^T\epsilon}{|X_1|^2}+o_p\left(\frac{1}{\sqrt{n}}\right)\right)\\
\label{mean.bound}\leq & \frac{|X_1|}{1+\sigma_n^2|X_1|^2}\left|\beta_1+\frac{X_1^T\epsilon}{|X_1|}+\sum_{i\geq 2}\gamma_i\beta_i\right|+\sum_{i\geq 2}\gamma_i(\beta_i-b_i)+o_p(1).
\end{align}
The first term is no longer random in $b$, and it can be made as small as we with now that it is decreasing in $\sigma_n$. If we set $\sigma_n^2\gg |\beta|_1\lambda_n/|X_1|$, this term is of order $o_p(1)$. 

\medskip

For the second term, we will apply lemma~\ref{lem:post.l1} to deduce that this term also goes to 0 in $\mathbb{P}_\beta\mu_Y^{b_{-1}}$ probability. To apply the posterior contraction result we need to establish the compatibility assumption~\eqref{eq:comp} on $W$. 

\begin{lemma}\label{lem:rec.to.comp}
Under assumption~\ref{assump.id},~\ref{assump.rec},~\ref{assump.dim} and the constraint $|X_1|=O(\sqrt{n})$, the matrix $W=(I-H)X_{-1}$ satisfies
\begin{equation*}
\kappa_0((2+\delta)s^*,W)=\inf_{|b|_0\leq (2+\delta)s^*}\frac{\sqrt{s^*}|Wb|}{\sqrt{n}|b|_1}\geq c
\end{equation*}
for some $c,\delta>0$.
\end{lemma}

We will prove the lemma after the proof of Theorem~\ref{main.th}. 

\medskip

To show~\eqref{var.conv}, Note that the integrand is not a random quantity. It suffices to show
\begin{equation*}
|\tau_n-1|=\left|\frac{\sigma_n^2|X_1|}{1+\sigma_n^2|X_1|^2}-\frac{1}{|X_1|}\right|\rightarrow 0.
\end{equation*}
That is certainly true for a $\{\sigma_n\}$ sequence chosen large enough. Combine~\eqref{mean.conv}, ~\eqref{var.conv} and the bound on the bounded Lipschitz distance, we have shown
\begin{equation*}
\mathbb{P}_\beta\left\lVert\mathcal{L}\left(|X_1|(b_1-\hat{\beta}_1)|Y\right)-\mathcal{N}(0,1)\right\rVert_{BL}\rightarrow 0.
\end{equation*}

\begin{proof}[Proof of lemma~\ref{lem:rec.to.comp}]
We will justify the compatibility assumption on $W$ in two steps. First we will show that the compatibility assumption of the $X$ matrix follows from the REC assumption~\ref{assump.rec}. Then we will show that the compatibility constant of $X$ and $W$ are not very far apart. 

\medskip

Let us first show that under assumption~\ref{assump.rec}, there exist constants $0<\delta<1$ and $c>0$, for which
\begin{equation*}
\kappa_0((2+\delta)s^*,X)=\inf_{|b|_0\leq (2+\delta)s^*}\frac{\sqrt{s^*}|Xb|}{\sqrt{n}|b|_1}\geq c.
\end{equation*}
Denote the support of $g$ as $S$. We have
\begin{align*}
\kappa_0((2+\delta)s^*, X) \geq & \inf_{|b|_0\leq (2+\delta)s^*}\frac{1}{\sqrt{2+\delta}}\frac{|Xb|}{\sqrt{n}|b_S|}\\
\geq & \min_{\substack{J\subset[p],\\ |J|\leq 3s^*}}\inf_{\substack{b\neq 0, \\ \left|b_{J^C}\right|_1\leq c_2|b_J|_1}}\frac{|Xb|}{\sqrt{n}|b_J|}\\
=& \kappa(3s^*,c_2)>0.
\end{align*}
Now, under assumptions~\ref{assump.id},~\ref{assump.rec} and~\ref{assump.dim}, we will show that there exist constants $0<\delta'<1$ and $c'>0$, for which
\begin{equation*}
\kappa_0((2+\delta')s^*,W)\geq \kappa_0((2+\delta)s^*,X)+o(1).
\end{equation*}
For $g\in\mathbb[R]^{p-1}$, we have
\begin{align*}
|Wg|= & \left|X\begin{bmatrix} 0\\g \end{bmatrix}-\sum_{i\geq 2}\gamma_ig_i\right|\\
\geq & \left|X\begin{bmatrix} 0\\g \end{bmatrix}\right|-\lambda_n|g_1|
\end{align*}
by assumption~\ref{assump.id}. Deduce that
\begin{align*}
\kappa_0((2+\delta')s^*,W) =& \inf_{|b|_0\leq (2+\delta)s^*}\frac{\sqrt{s^*}|Wb|}{\sqrt{n}|b|_1}\\
\geq & \kappa_0((2+\delta')s^*+1,X)-\sqrt{\frac{s^*}{n}}\lambda_n\\
=& \kappa_0((2+\delta')s^*+1,X)-\frac{\sqrt{s^*\log p}}{n}.
\end{align*}
The second term if order $o(1)$ under assumption~\ref{assump.dim}.

\end{proof}

\section*{Acknowledgement}
I would like to thank my advisor, Professor Pollard, for his expert advise and extraordinary support throughout this project.

\bibliographystyle{plain}
\bibliography{dana}
\end{document}